\documentclass[12pt]{article}

\usepackage{amsmath,amssymb,amsthm}
\usepackage{graphicx}
\usepackage[all]{xy}
\usepackage[margin=1in]{geometry}

\numberwithin{equation}{section}
\newtheorem{theorem}[equation]{Theorem}

\newtheorem{proposition}[equation]{Proposition}
\newtheorem{lemma}[equation]{Lemma}

\newtheorem{problem}[equation]{Problem}

\theoremstyle{definition}
\newtheorem{rmk}[equation]{Remark}

\newtheorem{eg}[equation]{Example}

\newtheorem{defn}[equation]{Definition}

\newcommand{\Q}{\mathbb{Q}}
\newcommand{\Z}{\mathbb{Z}}
\newcommand{\R}{\mathbb{R}}
\newcommand{\C}{\mathbb{C}}

\newcommand{\cO}{\mathcal{O}}
\newcommand{\fp}{\mathfrak{p}}

\newcommand{\cA}{\mathcal{A}}
\newcommand{\cH}{\mathcal{H}}
\newcommand{\eps}{\varepsilon}

\DeclareMathOperator{\girth}{girth}

\DeclareMathOperator{\sys}{sys\pi_1}
\DeclareMathOperator{\Char}{char}
\DeclareMathOperator{\SL}{SL}
\DeclareMathOperator{\PSL}{PSL}

\DeclareMathOperator{\PGL}{PGL}
\DeclareMathOperator{\ord}{ord}
\DeclareMathOperator{\im}{im}

\DeclareMathOperator{\Tr}{Tr}
\DeclareMathOperator{\Nm}{N}

\title{A Note on Riemann Surfaces of Large Systole}
\author{Shotaro Makisumi}
\date{\today}

\begin{document}
\maketitle
\abstract{We examine the \emph{large systole problem}, which concerns compact hyperbolic Riemannian surfaces whose systole, the length of the shortest noncontractible loops, grows logarithmically in genus. The generalization of a construction of Buser and Sarnak by Katz, Schaps, and Vishne, which uses principal ``congruence'' subgroups of a fixed cocompact arithmetic Fuchsian, achieves the current maximum known growth constant of $\gamma = \frac{4}{3}$. We prove that this is the best possible value of $\gamma$ for this construction using arithmetic Fuchsians in the congruence case. The final section compares the large systole problem with the analogous large girth problem for regular graphs.}

\section{Introduction}
Since their introduction, Riemann surfaces have proved a fruitful area of investigation bringing together different branches of mathematics. A notion of long-standing geometric interest is a surface's \emph{systole}, the length of its shortest noncontractible loops, or equivalently the length of the shortest closed geodesics. For hyperbolic Riemann surfaces, the lengths of closed geodesics in general have been studied, in part because of their relation to arithmetic. For overviews of various problems concerning closed geodesics of hyperbolic surfaces, see the surveys by Schmutz Schaller \cite{MR1609636} and by Parlier \cite{Par09}.

This paper examines the \emph{large systole problem}, on compact hyperbolic Riemannian surfaces whose systole grows logarithmically in genus. The best known constructions of such surfaces, achieving the growth constant $\gamma = \frac{4}{3}$, have all used arithmetic inputs. Our main result, Theorem \ref{theorem:MainResult}, is the sharpness of $\gamma = \frac{4}{3}$ for surfaces corresponding to principal ``congruence'' subgroups of a fixed cocompact arithmetic Fuchsian in the congruence case. The proof reduces the systole bound to a question about representation by an indefinite ternary quadratic form, which in turn is solved by a local-global approach.

\paragraph{The Large Systole Problem} \label{sec:SystoleProblem}
In this section, we define the large systole problem and discuss some of the known results on Riemann surfaces of large systole. For more details, see Parlier \cite[Section~4.1]{Par09}.

Throughout, a hyperbolic Riemann surface is viewed with its hyperbolic metric as a Riemannian surface. Let $\cH$ be the upper half plane equipped with the hyperbolic metric, on which $\SL_2(\R)$ (or $\PSL_2(\R)$) acts by fractional linear transformations. Let $\Gamma$ be a \emph{Fuchsian group}, that is, a discrete subgroup of $\SL_2(\R)$ (or of $\SL_2(\R)$). Then $\Gamma$ acts properly discontinuously on $\cH$, so that for $\Gamma$ acting on $\cH$ without a fixed point, the quotient $\Gamma\backslash\cH$ has the structure induced from $\cH$ of a (smooth) hyperbolic Riemann surface. Conversely, every hyperbolic Riemann surface (hereafter simply Riemann surface) arises as such a quotient.

\begin{defn}
 Let $X$ be a compact Riemann surface. The \emph{systole} of $X$, denoted $\sys(X)$, is the length of the shortest noncontractible loops of $X$.
\end{defn}
Equivalently, systole is the length of the shortest closed geodesics. It is this characterization that allows an arithmetic study of systole; see Section \ref{sec:Background}.

\begin{rmk}
 The notation $\sys$ denotes, more precisely, \emph{homotopy 1-systole}.

 Non-compact Riemann surfaces have arbitrarily short noncontractible loops around its cusps. For example, let $\Gamma$ be the unipotent group generated by $\left(\begin{smallmatrix}1 & 1 \\ 0 & 1\end{smallmatrix}\right)$, which acts on $\cH$ as horizontal translation by $1$. Then horizontal segments connecting $iy$ and $1 + iy$ become loops around the cusp at infinity in the non-compact quotient $\Gamma\backslash\cH$, with length $1/y \rightarrow 0$ as $y \rightarrow \infty$. For general Riemann surfaces, systole is therefore defined excluding these boundary curves.
\end{rmk}

We make the following definition which, while not standard in the literature, has an established counterpart for regular graphs; see Section \ref{sec:Girth}.
\begin{defn}
 A family of compact Riemann surfaces $X_i$ with $g(X_i) \rightarrow \infty$ is \emph{of large systole} if there exists constants $\gamma > 0$ and $c$, independent of $i$, such that $X_i$ satisfy
 \[
  \sys(X_i) \ge \gamma \log g(X_i) - c.
 \]
\end{defn}
In \cite{MR2780746}, the systole of a random Riemann surface was shown to be unrelated to its genus. In particular, random compact Riemann surfaces do not form families of large systole.
\begin{problem}[The large systole problem (for compact Riemann surfaces)]
 Determine the supremum of $\gamma$ such that there exists a family of compact Riemann surfaces $X_i$ of large systole with this $\gamma$.
\end{problem}
A simple area consideration gives the upper bound $\gamma \le 2$. Buser and Sarnak gave in \cite{MR1269424} the first construction of surfaces of large systole.
\begin{theorem}[Buser and Sarnak \cite{MR1269424}, Section 4]
Let $\Gamma$ be a cocompact Fuchsian in $\SL_2(\R)$ obtained from a quaternion division algebra $D(\frac{a, b}{\Q})$ for positive integers $a, b$. Let $\Gamma(p)$ for odd primes $p$ be the principal ``congruence'' subgroups. Then the family of compact Riemann surfaces $X_p = \Gamma(p)\backslash\cH$ satisfy $g(X_p) \rightarrow \infty$ as $p \rightarrow \infty$, and there exists some $c$, independent of $p$, such that
 \[
  \sys(X_p) \ge \frac{4}{3} \log g(X_p) - c.
 \]
\end{theorem}
This constant $\gamma = \frac{4}{3}$ remains the best achieved, so that the truth of the large systole problem lies somewhere in $\frac{4}{3} \le \gamma \le 2$. Section \ref{sec:Background} will provide all necessary definitions and describe a generalization to quaternion algebras over number fields by Katz, Schaps, and Vishne \cite{MR2331526}, which also achieves $\gamma = \frac{4}{3}$. Although a complete proof does not seem to be available in the literature, Parlier in \cite[to appear in Chapter~3]{Par09} notes that the Platonic surfaces of Brooks \cite{MR1677565}, which are compactifications of the non-compact Riemann surfaces corresponding to principal congruence subgroups of $\SL_2(\Z)$, also achieve $\gamma = \frac{4}{3}$. All known constructions of surfaces of large systole thus have arithmetic inputs, which is perhaps surprising given the seemingly inherent geometric nature of the problem.

In this direction, Schmutz Schaller has shown in \cite{MR1277506} that hyperbolic Riemann surfaces corresponding to principal congruence subgroups of $\SL_2(\Z)$ have the maximum systole among all surfaces in their respective Teichm\"{u}ller spaces. For compact surfaces, the only known result is in genus 2, where first Jenni \cite{MR749104} and since others have shown that the Bolza curve attains the maximum systole. Schmutz Schaller proved this in \cite[Theorem~5.3]{MR1250756} and also conjectured, with partial results in \cite[Section~8]{MR1250756}, that the so-called $M(3)$ surface attains the maximum systole in genus 3. Both surfaces correspond to arithmetic Fuchsians, the arithmetic triangle groups $(2, 3, 8)$ and $(2, 3, 12)$, respectively.

\paragraph{Contribution}
Jenni's and Schmutz Schaller's results indicate the possible central importance of arithmetic surfaces in the large systole problem, even for compact surfaces. A question that does not appear to have been considered in the literature is an upper bound on systole for arithmetic surfaces, and in particular whether $\gamma = \frac{4}{3}$ is optimal for the generalized Buser-Sarnak construction. This paper aims to address this gap by showing the following.
\begin{theorem}[Main Result (Sharpness of $\gamma = \frac{4}{3}$ for cocompact arithmetic Fuchsians in the congruence case)] \label{theorem:MainResult}
 The constant $\gamma = \frac{4}{3}$ is the best possible for the family of compact Riemann surfaces corresponding to principal ``congruence'' subgroups of a fixed arithmetic subgroup in the congruence case.

 More precisely, let $\Lambda$ be a cocompact arithmetic subgroup of $\SL_2(\R)$. Fix a totally real number field $K \subset \R$ with ring of integers $R$, a quaternion division algebra $A$ over $K$ split at the given embedding and ramified at all other infinite places, an isomorphism $\varphi: A \otimes_K \R \rightarrow M_2(\R)$, and an $R$-order $\cO$ in $A$, such that $\Lambda$ is commensurable with $\Gamma = \Gamma(\cO)$. Assume, moreover, that $\Lambda \cap \Gamma$ is a congruence subgroup of $\Gamma$. Then for any $\gamma > \frac{4}{3}$ and any sequence of ideals $I$ of $R$, the compact Riemann surfaces $X_I = \Lambda(I)\backslash\cH$ eventually satisfy
 \[
  \sys(X_I) < \gamma\log g(X_I).
 \]
\end{theorem}
This will be proved in Section \ref{sec:MainResult}. The last part of Section \ref{sec:Background} will discuss the ``congruence case'' condition and the obstruction to the general case.

\paragraph{Acknowledgments} This paper was written as the author's undergraduate thesis at Princeton University. It is a pleasure to thank Professor Peter Sarnak for his excellent guidance. I would also like to thank Professor Nicolas Templier for his helpful comments and for volunteering to be the second reader.

\section{Background} \label{sec:Background}
We recall some facts about quaternion algebras and arithmetic subgroups of $\SL_2(\R)$. See \cite{MR1937957} for details. We will also use some results from the arithmetic theory of quadratic forms. These can be found in \cite{MR522835} (over $\Q$) and \cite{MR1754311} (over global fields).

\paragraph{Quaternion algebras and arithmetic subgroups of $\SL_2(\R)$} \label{subsec:ArithmeticSubgroups}
Let $K$ be a field with $\Char K \neq 2$. A \emph{quaternion algebra} over $K$ is a central simple algebra of dimension 4 over $K$. Every quaternion algebra over $K$ is isomorphic to some $(\frac{a, b}{K})$ for $a, b \in K^\times$, i.e. $K$-algebra with basis $1, i, j, k$, where $1$ is the unit and $i, j, k$ satisfy $i^2 = a, j^2 = b$, and $ij = -ji = k$. Every quaternion algebra admits a canonical notion of conjugation, from which \emph{(reduced) trace} and \emph{(reduced) norm} are defined. For $x = x_0 + x_1i + x_2j + x_3k$ in $A = (\frac{a, b}{K})$, these are
\begin{align*}
 \Tr_A(x) &= 2x_0 \\
 \Nm_A(x) &= x_0^2 - ax_1^2 - bx_2^2 + abx_3^2,
\end{align*}
while in $M_2(K)$, the 2-by-2 matrices with entries in $K$, they agree with the matrix trace and norm. From the reduced norm, it follows that $A = (\frac{a, b}{K})$ is a division algebra if and only if the associated quadratic form $X_0^2 - aX_1^2 - bX_2^2 + abX_3^2$ is anisotropic over $K$, i.e. if it does not represent 0 for nonzero $(x_0, x_1, x_2, x_3) \in K^4$. If the form is isotropic, then $A \cong M_2(K)$, and $A$ is said to be \emph{split}.

Now let $K$ be a number field of degree $n$ with ring of integers $R$. A place (finite or infinite) of $K$ will be denoted by $\nu$. We write $\nu = \fp$, a prime ideal of $R$, or $\nu = \sigma$, an embedding $K \hookrightarrow \C$, to indicate finite and infinite places, respectively. Let $K_\nu$ be the completion of $K$ at $\nu$, and $R_\nu$ its ring of integers. Let $A$ be a quaternion algebra over $K$. Since the isomorphism class of $(\frac{a, b}{K})$ depends only on the class of $a$ and $b$ modulo $(K^\times)^2$, we can always present $A$ as $(\frac{a, b}{K})$ for $a, b \in R^\times$. For each place $\nu$ of $K$, $A_\nu := A \otimes_K K_\nu$ is a quaternion algebra over $K_\nu$, isomorphic to $M_2(K_\nu)$ or to a unique division algebra. Accordingly, $A$ is said to \emph{split} or to be \emph{ramified} at $\nu$.

To obtain subgroups of $\SL_2(\R)$, let now $K \subset \R$ be a real number field with distinct embeddings $\sigma_1, \ldots, \sigma_n$, where $\sigma_1$ is the trivial embedding, and let $A$ be a quaternion algebra over $K$ that splits at $\sigma_1$. Fix an isomorphism
\[
\varphi: A \otimes_K \R \longrightarrow M_2(\R).
\]
Throughout, groups in $A \subset A \otimes_K \R$ will be identified with their image via $\varphi$ in $M_2(\R)$. Since conjugation in a quaternion algebra was canonically defined, $\varphi$ takes reduced trace and norm to the matrix trace and norm.

Let $\cO$ be an \emph{($R$-)order} in $A$, i.e. a subring of $A$ that is also a finitely generated $R$-module such that $\cO \otimes_{R} K = A$. The norm-1 quaternions in $\cO$,
\[
 \cO^1 = \{ x \in \cO : \Nm_A(x) = 1 \},
\]
form a multiplicative subgroup whose image under $\varphi$ lies in $\SL_2(\R)$. By a theorem of Borel and Harish-Chandra, this image is discrete in $\SL_2(\R)$ if and only if $K$ is totally real (i.e. $\sigma_i : K \hookrightarrow \R$ for every $i$) and $A$ splits at exactly one infinite place. That is,
\[
 A_{\sigma_i} \cong
 \begin{cases}
  M_2(\R)    &\mbox{ if } i = 1 \\
  \mathbb{H} &\mbox{ if } i > 1,
 \end{cases}
\]
where $\mathbb{H} = (\frac{-1, -1}{\R})$ is the Hamiltonian quaternion algebra. In this case, we write $\Gamma(\cO)$ for $\cO^1$, identified with its image in $\SL_2(\R)$.

An \emph{arithmetic subgroup} is a subgroup $\Lambda$ of $\SL_2(\R)$ that is \emph{commensurable} with (the image in $\SL_2(\R)$ of) some $\Gamma = \Gamma(\cO)$; that is, $\Lambda \cap \Gamma$ has finite index in both $\Lambda$ and $\Gamma$. If $K = \Q$ and $A \cong M_2(\Q)$, then $\Lambda$ is not cocompact and is commensurable with some conjugate of $\SL_2(\Z)$. In all other cases, $\Lambda$ is cocompact.

\begin{rmk}
 There is an equivalent definition of arithmetic subgroups using a finite-dimensional real representation of $\SL_2(\R)$, and a yet more general definition for real Lie groups; see \cite[Appendix~to~Chapter~1,~Section~1]{MR1071179}. The equivalence of these definitions with the one above follows from the classification of classical groups by Weil.
\end{rmk}

\paragraph{``Congruence'' subgroups and the generalized Sarnak-Buser construction} \label{subsec:BuserSarnak}
Let $K$, $A$, $\varphi$, and $\cO$ be as in the previous section, and let $\Gamma = \Gamma(\cO)$. For an ideal $I$ of $R$, $I\cO$ is an ideal in $\cO$. Let $\pi: \cO \rightarrow \cO/I\cO$ be the quotient map. The \emph{principal ``congruence'' subgroup of $\Gamma$ modulo $I$}, denoted $\Gamma(I)$, is defined as
\[
 \Gamma(I) = \pi^{-1}(1 + I\cO) = \{ x \in \Gamma : x \equiv 1 \bmod I\cO \}
\]
and is identified with its image under $\varphi$ in $\SL_2(\R)$. A \emph{``congruence'' subgroup of $\Gamma$} is one containing some principal ``congruence'' subgroup of $\Gamma$.

Let $\Lambda \subset \SL_2(\R)$ be an arithmetic subgroup, and fix $K$, $A$, $\varphi$, and $\cO$ as before such that (the image in $\SL_2(\R)$ of) $\Gamma = \Gamma(\cO)$ is commensurable with $\Lambda$. The \emph{principal ``congruence'' subgroup of $\Lambda$ modulo $I$}, denoted $\Lambda(I)$ and depending on the choices of $K$, $A$, $\varphi$, and $\cO$, is defined as
\[
 \Lambda(I) = \Lambda \cap \Gamma(I),
\]
again viwed via $\varphi$ both in $\SL_2(\R)$ and as a subgroup of $\Gamma$. A \emph{``congruence'' subgroup of $\Lambda$} is one containing some principal ``congruence'' subgroup of $\Lambda$.

\begin{rmk}
 In the split case $\Gamma = \SL_2(\Z)$, this notion of (principal) congruence subgroups agree with the usual definition. We place ``congruence'' in quotation marks since these groups may not be congruence subgroups of $\SL_2(\Z)$ in the usual sense, i.e. they may not contain any of the usual principal congruence subgroups of $\SL_2(\Z)$. Already $\SL_2(\Z)$ contains such finite-index non-congruence subgroups.
\end{rmk}

Katz, Schaps, and Vishne \cite{MR2331526} showed the large systole property for principal ``congruence'' subgroups of a fixed arithmetic subgroup.
\begin{theorem}[Katz-Schaps-Vishne \cite{MR2331526}, Theorem 1.5, Generalized Buser-Sarnak construction] \label{theorem:KSV}
 For principal ``congruence'' subgroups of an arbitrary fixed cocompact arithmetic subgroup, the family of corresponding compact Riemann surfaces is of large systole with $\gamma = \frac43$.

 More precisely, let $\Lambda$ be a cocompact arithmetic subgroup of $\SL_2(\R)$, and fix $K$, $A$, $\varphi$, and $\cO$ such that $\Gamma = \Gamma(\cO)$ is commensurable with $\Lambda$. Then for $X_I = \Lambda(I)\backslash\cH$, as $I$ varies over any sequence of ideals of $R$ of sufficiently large norm, $g(X_I) \rightarrow \infty$, and there exists a constant $c$, independent of $I$, such that
 \[
  \sys(X_I) \ge \frac{4}{3}\log g(X_I) - c.
 \]
\end{theorem}
The original Buser-Sarnak construction corresponds to the case $K = \Q$, $A = (\frac{a, b}{\Q})$, where $a, b$ are positive integers, $\cO = \Z \oplus \Z i \oplus \Z j \oplus \Z k$, and $I = (p)$ for odd primes $p$; that is, writing $x = x_0 + x_1i + x_2j + x_3k$,
\begin{align*}
 \Gamma &= \{ x \in \cO : x_0^2 - ax_1^2 - bx_2^3 + abx_3^2 = 1 \} \\
 \Gamma(p) &= \{ x \in \Gamma : x_0 \equiv 1 \bmod p, \quad x_1 \equiv x_2 \equiv x_3 \equiv 0 \bmod p \}.
\end{align*}

The key fact relating systoles to arithmetic is the bijection between conjugacy classes of hyperbolic elements in a torsion-free Fuchsian group $\Gamma$ and closed geodesic of the corresponding Riemann surface. In one direction, a hyperbolic element $x \in \Gamma$ fixes as a set its \emph{axis}, the unique geodesic connecting its two fixed points on the boundary of $\cH$. For any point $P$ on this axis, the geodesic from $P$ to $x\cdot P$ descends to a closed geodesic in $\Gamma\backslash\cH$ depending only on the conjugacy class of $x$. Moreover, the length $l_x$ of this loop satisfies $e^{l_x/2} = \lambda$, where $\lambda$ is the unique positive real such that $x$ is conjugate in $\SL_2(\R)$ to $\left(\begin{smallmatrix}\lambda & 0 \\ 0 & \lambda^{-1}\end{smallmatrix}\right)$, so $|\Tr_{M_2(\R)}(x)| = \lambda + \lambda^{-1}$. If $\Gamma = \Gamma(\cO)$ as before, then $\Tr_A(x) = 2x_0$ for any presentation $x = x_0 + x_1i + x_2j + x_3k$. In both the Buser-Sarnak construction and its generalization, a trace 
estimate produces the systole bound. Katz, Schaps, and Vishne's lower bound on trace shows in particular that, for all $\Nm(I)$ large, $\Gamma(I)$ consists of hyperbolic elements and so acts with no fixed point, so that $X_I$ is indeed a Riemann surface.

\paragraph{Congruence condition and obstructions to the general case}
Katz, Schaps, and Vishne's result applies to arbitrary cocompact arithmetic subgroups $\Lambda$, while our Theorem \ref{theorem:MainResult} assumes the congruence case, i.e. that $\Lambda \cap \Gamma$ is a congruence subgroup of $\Gamma$. We briefly remark here on this difference.

To begin, since $\Lambda(I) = \Gamma(I) \cap \Lambda$, we may as well assume by replacing $\Lambda$ by $\Lambda \cap \Gamma$ that $\Lambda$ is a finite-index subgroup of $\Gamma$.  It is not hard to see that $[\Gamma(I) : \Lambda(I)] \le [\Gamma : \Lambda]$.

Recall that $g(\Gamma\backslash\cH) = \frac{1}{4\pi}\mu(\Gamma\backslash\cH) + 1$ for a torsion-free Fuchsian $\Gamma$, where $\mu$ is the hyperbolic measure, and that $\mu(\Gamma_2\backslash\cH) = [\Gamma_1 : \Gamma_2]\mu(\Gamma_1\backslash\cH)$ for $\Gamma_2 \subset \Gamma_1$ of finite index. In our case, $g(X_I) = g(\Lambda(I)\backslash\cH)$ and $g(\Gamma(I)\backslash\cH)$ therefore differ by a constant factor bounded uniformly in $I$, which merely affects the undetermined constant $c$ in the large systole problem.

Going to a finite-index subgroup in the Fuchsian corresponds to taking a finite cover of the surface, so all geodesic loops and hence the systole can only increase. Katz, Schaps, and Vishne therefore only needed to prove their result for $\Gamma(I)$. Meanwhile, even for an index-two subgroup, it is not clear that a shortest geodesic loop splits in the cover, which is necessary to carry out this reduction for our result.

The congruence subgroup problem---whether every finite subgroup is a congruence subgroup---is known to be false for the case in question. For background on the congruence subgroup problem, see for example \cite{PR} and references therein.

\section{Proof of Main Result} \label{sec:MainResult}
Let $\Lambda$, $K$, $A$, $\varphi$, and $\cO$ be given as in Theorem \ref{theorem:MainResult}. It is equivalent to show that, for any $\gamma > \frac{4}{3}$, the expression
\[
 \sys(X_I) - \gamma \log g(X_I)
\]
assumes arbitrarily large negative values as $I$ varies over any sequence of ideals of $R$, i.e. tends to $-\infty$ as $\Nm(I) \rightarrow \infty$. First considering $\Gamma(I)$ (that is, without reference to $\Lambda$), we bound $\sys(\Gamma(I)\backslash\cH)$ and $g(\Gamma(I)\backslash\cH)$ in terms of $\Nm(I)$. For the genus, we will show the following.
\begin{proposition} \label{prop:GenusBound}
 For $K$, $A$, $\varphi$, and $\cO$ as in Theorem \ref{theorem:MainResult}, there exists a constant $C_1 > 0$ such that, for all ideals $I$ of $R$, we have
 \[
  g(\Gamma(I)\backslash\cH) > C_1\Nm(I)^3.
 \]
\end{proposition}
For the systole bound, note that if $\Gamma(I)$ contains $x \neq \pm 1$, then $|\Tr_A(x)| = \lambda + \lambda^{-1} > \lambda = e^{l_x/2}$, so
\[
 \sys(\Gamma(I)\backslash\cH) < 2\log(|\Tr_A(x)|).
\]We will show the following.
\begin{proposition} \label{prop:SystoleBound}
 For $K$, $A$, $\varphi$, and $\cO$ as in Theorem \ref{theorem:MainResult}, there exists a constant $C_2 > 0$ such that, for all ideals $I$ of $R$, $\Gamma(I)$ contains some $x \neq \pm 1$ satisfying
\[
 |\Tr_A(x)| < C_2\Nm(I)^2.
\]
\end{proposition}
Theorem \ref{theorem:MainResult} follows from these two bounds. Indeed, as discussed at the end of the previous section, $g(X_I) = g(\Lambda(I)\backslash\cH)$ and $g(\Gamma(I)\backslash\cH)$ differ at most by a constant factor bounded uniformly in $I$, so Proposition \ref{prop:GenusBound} implies that
\[
 g(X_I) > C_1'\Nm(I)^3
\]
for some new constant $C_1' > 0$ independent of $I$. The reduction for the systole bound uses our assumption that $\Lambda \cap \Gamma$ is a congruence subgroup of $\Gamma$. Say $\Lambda \cap \Gamma \supset \Gamma(I_0)$, $I_0$ an ideal of $R$. Then for any ideal $I$ of $R$,
\[
 \Lambda(I) = (\Lambda \cap \Gamma) \cap \Gamma(I) \supset \Gamma(I_0) \cap \Gamma(I) \supset \Gamma(I_0I),
\]
so Proposition \ref{prop:SystoleBound} applied to the ideal $I_0I$ implies
\[
 \sys(X_I) < 2\log(C'_2N(I)^2),
\]
where $C'_2 = C_2 N(I_0)^2 > 0$ is again independent of $I$. Hence
\begin{align*}
 \sys(X_I) - \gamma \log g(X_I) &< 2\log(C'_2\Nm(I)^2) - \gamma \log(C'_1\Nm(I)^3) \\
 &= (4 - 3\gamma)\log \Nm(I) + 2 \log C'_2 - \gamma\log C'_1,
\end{align*}
and as $\Nm(I) \rightarrow \infty$, the last expression tends to $-\infty$ whenever $\gamma > \frac{4}{3}$.

\subsection{Genus bound}
Katz, Schaps, and Vishne proved an upper bound for $g(\Gamma(I)\backslash\cH)$ in terms of $\Nm(I)$ in \cite[Section~4,~Corollary~4.6]{MR2331526}. We follow their approach to prove a lower bound.
\begin{proof}[Proof of Proposition \ref{prop:GenusBound}]
 For all large $\Nm(I)$, $\Gamma(I)$ is torsion-free, so $g(\Gamma(I)\backslash\cH) = \frac{1}{4\pi}\mu(\Gamma(I)\backslash\cH) + 1 = \frac{1}{4\pi}[\Gamma : \Gamma(I)]\mu(\Gamma\backslash\cH) + 1$. It therefore suffices to give a similar bound for $[\Gamma : \Gamma(I)]$.

 For an order $\cO$ in a quaternion algebra $A$ over a number field $K$, \cite[Section~4]{MR2331526} defined for each ideal $I$ of $R$ a norm map
 \begin{align*}
  \Nm_{\cO/I\cO} : \cO/I\cO &\longrightarrow R/I \\
       x + I\cO &\longmapsto \Nm_A(x) + I,
 \end{align*}
 so that
 \[
  \xymatrix{
   \cO \ar[r]^{\Nm_A} \ar[d]_\pi   & R \ar[d]^{\pi_0} \\
   \cO/I\cO \ar[r]_{\Nm_{\cO/I\cO}} & R/I
  }
 \]
 commutes, and defined
 \[
  (\cO/I\cO)^1 = \Nm_{\cO/I\cO}^{-1}(1 + I) = \{ x + I\cO : x \in \cO, \Nm_A(x) = 1 \bmod I \}.
 \]
 Then $\pi$ takes $\Nm_A^{-1}(1)$ into $\Nm_{\cO/I\cO}^{-1}(1 + I)$, i.e.
 \[
  \pi|_\Gamma : \Gamma \rightarrow (\cO/I\cO)^1.
 \]
 This map is surjective. Indeed, for lattices in regular quadratic spaces over $n \ge 4$ variables that are isotropic in at least one completion, everywhere local representation implies global representation, and the global solution can be taken arbitrarily close to local solutions at finitely many places; see \cite[Theorem~104:3]{MR1754311}. If $x' \in (\cO/I\cO)^1$, then ${x'_0}^2 - a{x'_1}^2 - b{x'_2}^2 + ab{x'_3}^2 \equiv 1 \bmod I$. Here the norm form splits at $\sigma_1$, so applying this result to local solutions $x'$ at primes dividing $I$, there exists $x \in \cO$ such that $x_0^2 - ax_1^2 - bx_2^2 + abx_3^2 = 1$ and $x \equiv x' \bmod I\cO$.

 By definition, $\Gamma(I)$ is the kernel of $\pi|_\Gamma$, so surjectivety implies
 \begin{align} \label{eq:index}
  [\Gamma : \Gamma(I)] = \#(\cO/I\cO)^1.
 \end{align}
 Let $I = \prod \fp_i^{t_i}$ be the prime ideal factorization. By the Chinese remainder theorem, $\cO/I\cO \cong \prod \cO/\fp_i^{t_i}\cO$. Since each projection preserves the norm,
 \begin{align} \label{eq:CRT}
  (\cO/I\cO)^1 \cong \prod(\cO/\fp_i^{t_i}\cO)^1.
 \end{align}
 For a prime ideal $\fp$ of $R$, let $\cO_\fp = \cO \otimes_R R_\fp$. By \cite[Lemma~4.2]{MR2331526}, $\cO/\fp^t\cO \cong \cO_\fp/\fp^t\cO_\fp$, so we are reduced to local calculations.

 For all but a finite set $P$ of primes, $\cO_\fp$ is a maximal $R_\fp$-order in $A_\fp$. In this case, \cite[Lemma~4.3]{MR2331526} implies
 \begin{align} \label{eq:notP}
  \#(\cO_\fp/\fp^t\cO_\fp)^1 > \Nm(\fp)^{3t} \text{ if } \fp \notin P.
 \end{align}
 In general, since $(\cO_\fp/\fp^t\cO_\fp)^1$ is the kernel of the induced norm map $\nu: (\cO_\fp/\fp^t\cO_\fp)^\times \rightarrow (R/\fp^t)^\times$,
 \[
  \#(\cO_\fp/\fp^t\cO_\fp)^1 = \frac{\#(\cO_\fp/\fp^t\cO_\fp)^\times}{\#\im(\nu)} > \frac{\#(\cO_\fp/\fp^t\cO_\fp)^\times}{\#(R/\fp^t)}.
 \]
 The denominator is $\Nm(\fp)^t$. In $\cO_\fp/\fp^t\cO_\fp$, the unity plus any element of the nilpotent ideal $\fp\cO_\fp/\fp^t\cO_\fp$ is invertible, so $\#(\cO_\fp/\fp^t\cO_\fp)^\times \ge \Nm(\fp)^{4(t-1)}$. Hence
 \begin{align} \label{eq:P}
  \#(\cO_\fp/\fp^t\cO_\fp)^1 > \Nm(\fp)^{3t-4} \text{ if } \fp \in P.
 \end{align}
 By \eqref{eq:index} through \eqref{eq:P},
 \[
  [\Gamma : \Gamma(I)] = \prod \#(\cO_{\fp_i}/\fp_i^{t_i}\cO_{\fp_i})^1 > \prod_{\fp_i \in P}\Nm(\fp_i)^{-4} \cdot \prod \Nm(\fp_i^{t_i})^3 \ge C_1 \Nm(I)^3,
 \]
 where
 \[
  C_1 = \prod_{\fp \in P}\Nm(\fp)^{-4}. \qedhere
 \]
\end{proof}
The constant can easily be improved, though this is not necessary for our purpose. For example, the quotient $(\cO_\fp/\fp^t\cO_\fp)/(\fp\cO_\fp/\fp^t\cO_\fp) \cong \cO_\fp/\fp\cO_\fp$ is an algebra over the field $R_\fp/\fp R_\fp$ of $\Nm(\fp)$ elements, so contains at least $\Nm(\fp) - 1$ invertible elements, the nonzero multiples of the unity.

\subsection{Systole bound}
We reduce Proposition \ref{prop:SystoleBound} to a question about representation by a quadratic form, for which local-global results apply.

We sketch the argument for the case of the original Buser-Sarnak construction---$A = (\frac{a, b}{\Q})$, where $a$ and $b$ positive integers, $\cO = \Z \oplus \Z i \oplus \Z j \oplus \Z k$, and prime ideals $I = (q)$---which will readily generalize to and illustrate our approach for the general case. If $x_0 + x_1i + x_2j + x_3k \in \Gamma(q)$, then $x_i \in \Z$ satisfy
\[
 x_0^2 - ax_1^2 - bx_2^2 + abx_3^2 = 1, \quad x_0 \equiv 1 \bmod q, \quad x_1 \equiv x_2 \equiv x_3 \equiv 0 \bmod q.
\]
Then $x_0^2 \equiv 1 \bmod q^2$, which for $q \neq 2$ implies $x_0 \equiv 1 \bmod q^2$. Set $x_0 = nq^2 + 1$ for some nonzero $n \in \Z$, and $x_i = qy_i$ for $i > 0$. Then
\begin{align*}
 (nq^2 + 1)^2 - aq^2y_1^2 - bq^2y_2^2 + abq_3^2 &= 1 \\
 \Leftrightarrow ay_1^2 + by_2^2 - aby_3^2 &= n^2q^2 + 2n
\end{align*}
by dividing through by $q^2$. Proposition \ref{prop:SystoleBound} claims the bound $2|nq^2 + 1| < C_2q^2$, so given any $q$, we must find some small nonzero $n$ (bounded by a constant) such that $n^2q^2 + 2n$ is represented by the indefinite ternary quadratic form $aY_1^2 + bY_2^2 - abY_3^2$ over $\Z$. We use a local-global approach. For primes $p \nmid 2ab$, the following lemma, also used in the general case, ensures local representability for any choice of $n$.
\begin{lemma} \label{lemma:unitUniversal}
 Let $F$ be a non-dyadic (non-archimedean) local field with valuation ring $R$ and unique maximal ideal $\fp$. If $a_1, a_2, a_3 \in R^\times$, then $a_1Y_1^2 + a_2Y_2^2 + a_3Y_3^2$ is universal over $R$.
\end{lemma}
\begin{proof}
 Let $k$ be the residue field of $F$, and denote reduction to $k$ by a bar. Given $x \in R$, since $a_2, a_3 \in R^\times$, the subsets
 \[
  \{ \overline{a_1} + \overline{a_2}(\overline{y_2})^2 : \overline{y_2} \in k \}, \quad \{ \overline{x} - \overline{a_3}(\overline{y_3})^2 : \overline{y_3} \in k \}
 \]
 of $k$ both have order $(\#k + 1)/2$ and so must overlap. Thus, lifting to $R$, there exist $y_2, y_3 \in R$ such that $a_1 + a_2y_2^2 + a_3y_3^2 - x \in \fp$. Consider
 \[
  f(X) = a_1X^2 + a_2y_2^2 + a_3y_3^2 - x \in R[X], \quad f'(X) = 2a_1X.
 \]
 Then $f(1) = a_1 + a_2y_2^2 + a_3y_3^2 - x \in \fp$ and $f'(1) = 2a_1 \notin \fp$, so by Hensel's lemma, there exists $y_1 \in R$ such that $a_2y_1^2 + a_2y_2^2 + a_3y_3^2 = x$.
\end{proof}
At the remaining finitely many primes $p$, we find a congruence condition (depending on $q$) on $n$ modulo a power (independent of $q$) of $p$ that ensures the local representability of $n^2q + 2n$. By the Chinese remainder theorem, there exists some $n$ satisfying all these congruences, and we take $C_2$ independent of $q$ as a bound for say $4|n|$.

We mentioned earlier that local-global principle holds for indefinite (isotropic in at least one place) forms in four or more variables over the ring of integers $R$ of any number field $K$ (more generally, lattices in a quadratic space over a global field). In general, the strong approximation theorem for the spin group implies that the spinor genus of an indefinite form in three or more variables consists of a single class, thus reducing the question of representation by such a form to representation by its spinor genus. Although indefinite forms in three variables may have \emph{spinor exceptions}---elements of $R$ represented by the form's genus but not by its spinor genus---they have been characterized by Kneser, Hsia, and Schulze-Pillot. We only need the following, which is essentially already in Kneser's work \cite[Remark~after~Satz~2]{MR0140487} (observed to hold in the generality we require by Schulze-Pillot \cite[Satz~1,~Bemerkung~1]{MR599822}): all spinor exceptions of a ternary quadratic form 
lies in finitely many square classes. That is, there exist $t_1, \ldots, t_r \in R^\times$ such that local-global principle holds for $x \in R$ not lying in any $t_iR^2$. We can impose additional congruence conditions on $n$ to ensure that $n^2q^2 + 2n$ avoids these square classes, hence is globally representable.

Although $R$ may not be a PID in the general case, it will still suffice to consider elements of $\Gamma(I)$ of the form
\[
(\alpha (\kappa \beta)^2 + 1) + (\kappa \beta y_1)i + (\kappa \beta y_2)j + (\kappa \beta y_3)k,
\]
where now $\beta \in I$ is a nonzero element of small norm, and $\kappa \in R$ is such that $\cO$ contains $R \oplus \kappa R i \oplus \kappa R j \oplus \kappa R k$. As above, this reduces to the local representability of $\alpha^2(\kappa \beta)^2 + 2\alpha$ by the associated ternary form. At finite places, this is again guaranteed by congruences at finitely many primes. For $K \neq \Q$, representability at nontrivial infinite places depends on $\sigma\alpha$ for $\sigma \neq \sigma_1$, and $|\alpha|$ can no longer be bounded independent of $\Nm(I)$. A quantitative form of (classical) strong approximation, following the proof suggested by Kneser in \cite[Chapter~1,~Section~15]{MR911121}, will provide the desired bound for $|\alpha|$.

Denote the ad\`{e}le ring of $K$ by $K_\cA$, and as usual identify $K \subset K_\cA$ by the diagonal embedding. For this quantative strong approximation, we will use the following results from \cite[Chapter~1]{MR911121}.
\begin{lemma}[Corollary 1 in \cite{MR911121}, Chapter 1, Section 14] \label{lemma:cassels1}
 There exists subset $W \subset K_\cA$ of the form
 \[
  W = \{ \xi = (\xi_\nu)_\nu \in K_\cA : |\xi_\nu|_\nu \le \delta_\nu \},
 \]
 where $\delta_\nu = 1$ for almost all $\nu$, such that
 \[
  K_\cA = W + K.
 \]
\end{lemma}
\begin{lemma}[Lemma in \cite{MR911121}, Chapter 1, Section 14] \label{lemma:cassels2}
 There exists a constant $C_K > 0$ depending only on $K$ such that, whenever $\xi = (\xi_\nu) \in K_\cA$ with
 \[
  \prod_\nu|\xi_\nu|_\nu > C_K,
 \]
 there exists $\lambda \in K^\times$ such that
 \[
  |\lambda|_\nu \le |\xi_\nu|_\nu \text{ for all } \nu.
 \]
\end{lemma}
We are now ready to prove the systole bound.

\begin{proof}[Proof of Proposition \ref{prop:SystoleBound}]
Fix a presentation $A = (\frac{a, b}{K})$, where $a, b \in R^\times$. Taking common denominator, there exists $\kappa \in R^\times$ such that $\kappa i, \kappa j, \kappa k \in \cO$, so
\[
 \cO \supset R \oplus \kappa R i \oplus \kappa R j \oplus \kappa R k.
\]

Let $P$ be the finite set of prime ideals of $R$ that divide $2ab$. Let $t_i \in R^\times$ be representatives of finitely many square classes containing all spinor exceptions of the indefinite ternary form $aY_1^2 + bY_2^2 - abY_3^2$ over $R$. Choose a prime ideal $\fp_0$ of $R$ not in $P$ and not dividing any $t_i$, and choose some $\alpha'_{\fp_0} \in \fp_0 \setminus \fp_0^2$.

Given an ideal $I$ of $R$, choose $\beta \in I$, to be refined later. The proof proceeds in three steps.
\begin{description}
 \item[Representability.] There exist $\eps_\fp > 0$ for $\fp \in P \cup \{\fp_0\}$, independent of $I$ (and of $\beta$), such that, for any $I$ and $\beta$, we can choose $\alpha'_\fp \in R_\fp$ for each $\fp \in P$ such that the following conditions on $\alpha \in R$ ensure the representability of $\alpha^2(\kappa \beta)^2 + 2\alpha$ by $aY_1^2 + bY_2^2 - abY_3^2$ over $R$.
  \begin{align}
   \left|\alpha - \left(-\frac{1}{(\kappa\beta)^2}\right) \right|_\sigma &< \left|\frac{1}{(\kappa\beta)^2}\right|_\sigma \text{ for all infinite } \sigma \neq \sigma_1 \label{eq:infiniteLocalInequality} \\
   |\alpha - \alpha'_\fp|_\fp &\le \eps_\fp \text{ for all } \fp \in P \label{eq:finiteLocalInequality} \\
   |\alpha - \alpha'_{\fp_0}|_{\fp_0} &\le \eps_{\fp_0}. \label{eq:spinorExceptionInequality}
  \end{align}
  The first two conditions ensure local representability at infinite and finite places, respectively, while the last condition avoids spinor exceptions.
 \item[Strong approximation.] There exists $C_{\sigma_1} > 0$ independent of $I$ such that, for any $I$ and $\beta$, there exist $\alpha \in R$ satisfying the conditions above and additionally
  \begin{align} \label{eq:givenInfiniteInequality}
   |\alpha| < C_{\sigma_1}\prod_{\sigma \neq \sigma_1} (\sigma(\kappa\beta))^2.
  \end{align}
  This will follow from a quantitative version of (the classical) strong approximation theorem.
 \item[End game.] Choose $\beta \in I$ of small norm by the Minkowski bound, then use $\beta$ and $\alpha$ obtained above to produce $x \in \Gamma(I)$ with the desired trace bound.
\end{description}
We take each step in turn.

\textbf{Representability.}
Let $\pi$ be the uniformizer of $R_{\fp_0}$. Choose $\eps_{\fp_0} > 0$ such that \eqref{eq:spinorExceptionInequality} implies
\begin{align} \label{eq:spinorExceptionCongruence}
 \alpha \equiv \alpha'_{\fp_0} \bmod \pi^2.
\end{align}

For $\fp \in P$, we will show that $aY_1^2$ alone represents $\alpha^2(\kappa\beta)^2 + 2\alpha$. Taking some $x_\fp \in R_\fp$ as the approximate root in Hensel's lemma, we will need $ax_\fp^2$ to be close $\fp$-adically close to $\alpha^2(\kappa \beta)^2 + 2\alpha$. Given $I$ and $\beta$, we will choose $\alpha'_\fp \in R_\fp$ to be a root of the quadratic $\alpha^2(\kappa \beta)^2 + 2\alpha - ax_\fp^2$ in $\alpha$, so
\[
 \alpha'_\fp = \frac{-1 \pm \sqrt{1 + (\kappa \beta)^2 ax_\fp^2}}{(\kappa \beta)^2}.
\]
Since the allowed error $\eps_\fp$ depends on $|x_\fp|_\fp$, we will choose $x_\fp \in R_\fp$ independent of $I$ and so that the above choice of $\alpha'_\fp$ is possible for any $I$ and $\beta$.

Let $\fp \in P$. Choose $x_\fp \in R_\fp$ with $|x_\fp|_\fp$ small enough that
\begin{align} \label{eq:valp}
 |a x_\fp^2|_\fp < |2|_\fp^2,
\end{align}
and take $\eps_\fp \in |K_\fp^\times|_\fp$ satisfying
\begin{align} \label{eq:epsp}
 \eps_\fp < |2ax_\fp|_\fp^2.
\end{align}
Given $I$ and $\beta$, for each $\fp \in P$, we choose $\alpha'_\fp$ as follows. Consider
\[
 f(X) = X^2 - (1 + (\kappa \beta)^2ax_\fp^2) \in R_\fp[X], \quad f'(X) = 2X.
\]
Then \eqref{eq:valp} implies $|f(1)|_\fp \le |ax_\fp^2|_\fp < |f'(1)|_\fp^2$, so by Hensel's lemma, there exists a unique $s_\fp \in R_\fp$ satisfying
\[
 s_\fp^2 = 1 + (\kappa\beta)^2ax_\fp^2, \quad |1 - s_\fp|_\fp < |2|_\fp.
\]
Since $-s_\fp$ is the other root of $f$, uniqueness implies $|1 + s_\fp|_\fp \ge |2|_\fp \ge |2|_\fp^2$. Then again by \eqref{eq:valp}, $|1 + s_\fp|_\fp > |ax_\fp^2|_\fp$, so
\[
 \frac{1 - s_\fp}{(\kappa\beta)^2} = \frac{1 - s_\fp^2}{(\kappa\beta)^2(1 + s_\fp)} = \frac{-ax_\fp^2}{1 + s_\fp} \in R_\fp.
\]
Taking this to be $\alpha'_\fp$, we obtain
\begin{align} \label{eq:equality}
 {\alpha'_\fp}^2(\kappa\beta)^2 + 2\alpha'_\fp = ax_\fp^2.
\end{align}

Now suppose $\alpha \in R$ satisfies \eqref{eq:infiniteLocalInequality} through \eqref{eq:spinorExceptionInequality} for these $\alpha'_\fp$. We show that $aY_1^2 + bY_2^2 - abY_3^2$ represents $\alpha^2(\kappa\beta)^2 + 2\alpha$ over $R$. First, \eqref{eq:spinorExceptionInequality} implies that $\alpha^2(\kappa\beta)^2 + 2\alpha$ is not a spinor exception. Indeed, since $\alpha_0 \in \fp_0 \setminus \fp_0^2$, \eqref{eq:spinorExceptionCongruence} implies $\ord_{\fp_0}(\alpha) = 1$. Then since $2 \notin \fp_0$, $\alpha (\kappa\beta)^2 + 2 \notin \fp_0$, so $\ord_{\fp_0}(\alpha^2 (\kappa\beta)^2 + 2\alpha) = 1$. Thus $\fp_0$ appears in the square-free part of the prime ideal factorization of the principal ideal $(\alpha^2(\kappa\beta)^2 + 2\alpha)R$. Meanwhile, any spinor exception has the form $t_ix^2$ for some $x \in R$, and $\fp_0$ by its choice does not divide the square-free part of $t_ix^2R = (t_iR)(xR)^2$.

It therefore suffices to show everywhere local representability. For infinite places $\sigma$, we must show that $(\sigma a)Y_1^2 + (\sigma b)Y_2^2 - (\sigma a)(\sigma b)Y_3^2$ represents $(\sigma\alpha)^2(\sigma(\kappa\beta))^2 + 2(\sigma\alpha)$ over $\R$. For $\sigma = \sigma_1$, the form is indefinite, hence universal. Otherwise, $A \otimes_\sigma \R \cong (\frac{\sigma a, \sigma b}{\R})$ is not split, so $\sigma a, \sigma b < 0$. The form is therefore negative definite, but \eqref{eq:infiniteLocalInequality} implies
\begin{align*}
 -\frac{2}{\sigma(\kappa\beta)^2} < \sigma\alpha < 0 \Longrightarrow (\sigma\alpha)^2(\sigma(\kappa\beta))^2 + 2(\sigma\alpha) < 0.
\end{align*}

For finite places, by Lemma \ref{lemma:unitUniversal}, we only need to check representability at primes in $P$. Let $\fp \in P$, and consider
\[
 f(X) = aX^2 - (\alpha^2(\kappa\beta)^2 + 2\alpha) \in R_\fp[X], \quad f'(X) = 2aX.
\]
Then by \eqref{eq:equality},
\begin{align*}
 |f(x_\fp)|_\fp &= |ax_\fp^2 - (\alpha^2(\kappa\beta)^2 + 2\alpha)|_\fp \\
                &= |({\alpha'_\fp}^2(\kappa\beta)^2 + 2\alpha'_\fp) - (\alpha^2(\kappa\beta)^2 + 2\alpha)|_\fp \\
                &\le |\alpha'_\fp - \alpha|_\fp.
\end{align*}
By \eqref{eq:finiteLocalInequality} and \eqref{eq:epsp}, this is less than $|2ax_\fp|_\fp^2 = |f'(x_\fp)|_\fp^2$. Hence by Hensel's lemma, already $aY_1^2$ represents $\alpha^2(\kappa\beta)^2 + 2\alpha$ over $R_\fp$.

\textbf{Strong approximation.} Define $\alpha' = (\alpha'_\nu) \in K_\cA$ by
\[
 \alpha'_\nu = 
 \begin{cases}
  -\frac{1}{\sigma(\kappa\beta)^2} &\mbox{ if } \nu = \sigma \neq \sigma_1 \\
  \alpha'_\fp                      &\mbox{ if } \nu = \fp \in P \cup \{\fp_0\} \\
  0                                &\mbox{ otherwise.}
 \end{cases}
\]
The existence of $\alpha \in R$ satisfying \eqref{eq:infiniteLocalInequality} through \eqref{eq:spinorExceptionInequality} follows from the classical strong approximation theorem, which asserts the density of $K$ in the restricted ad\`{e}les excluding $\sigma_1$. To obtain \eqref{eq:givenInfiniteInequality}, we additionally bound $|\alpha - \alpha'|_{\sigma_1} = |\alpha|$ by quantifying strong approximation.

Let $W \subset K_\cA$ and $\delta_\nu > 0$ be as in Lemma \ref{lemma:cassels1}; we may as well take $\delta_\nu \in |K_\nu^\times|_\nu$. Define $\lambda_\nu > 0$ for each place $\nu$ of $K$, as follows. For $\nu \neq \sigma_1$, let
\[
 \lambda_\nu = 
 \begin{cases}
  \frac12\frac{1}{(\sigma\beta)^2}\delta_\sigma^{-1} &\mbox{if } \nu = \sigma \neq \sigma_1 \\
  \eps_\fp \delta_\fp^{-1}                           &\mbox{if } \nu = \fp \in P \cup \{ \fp_0 \} \\
  \delta_\nu^{-1}                                    &\mbox{if } \nu \mbox{ finite and not in } P \cup \{ \fp_0 \}.
 \end{cases}
\]
Since $\eps_\fp, \delta_\nu$ are independent of $I$, there exists $C_{\sigma_1} > 0$ independent of $I$ such that, letting
\[
 \lambda_{\sigma_1} = C_{\sigma_1} \prod_{\sigma \neq \sigma_1}(\sigma\beta)^2\delta_{\sigma_1}^{-1},
\]
we have $\prod_\nu \lambda_\nu > C_K$ for $C_K > 0$ as in Lemma \ref{lemma:cassels2}. Since $\delta_\nu \in |K_\nu^\times|_\nu$ and $\eps_\fp \in |K_\fp^\times|_\fp$, for all $\nu$ there exists $\zeta_\nu \in K_\nu$ with $|\zeta_\nu|_\nu = \lambda_\nu$. Then $\zeta = (\zeta_\nu) \in K_\cA$ since $\lambda_\nu = \delta_\nu^{-1} = 1$ for almost all $\nu$, and $\prod_\nu|\zeta_\nu|_\nu > C_K$, so by Lemma \ref{lemma:cassels2} there exists $\lambda \in K^\times$ such that
\[
 |\lambda|_\nu \le |\zeta_\nu|_\nu = \lambda_\nu \mbox{ for all } \nu.
\]
Apply Lemma \ref{lemma:cassels1} to $\lambda^{-1}\alpha'$, then multiply by $\lambda$, so that $\alpha' = \lambda \xi + \alpha$ for some $\xi \in W$, $\alpha \in K$. Then at the $\nu$ component,
\[
 |\alpha - \alpha'_\nu|_\nu = |\lambda|_\nu|\xi_\nu|_\nu \le \lambda_\nu \delta_\nu.
\]
Taking $\nu$ to be each place of interest, we conclude that $\alpha$ satisfies \eqref{eq:infiniteLocalInequality} through \eqref{eq:givenInfiniteInequality}, and then in fact $\alpha \in R$ as above. Note that
\[
 |\alpha||\kappa\beta|^2 \le C_{\sigma_1}\prod_\sigma \sigma(\kappa\beta)^2 = C_{\sigma_1}\Nm(\kappa)^2\Nm(\beta)^2.
\]

\textbf{End game.} Recall that Minkowski's convex body theorem gives the following: there exists $B_K > 0$ depending only on $K$ such that every ideal $I$ of $R$ contains some nonzero element $\beta$ with $\Nm(\beta) \le B_K \Nm(I)$. Given $I$, choose such a $\beta$. By the above, there exists nonzero $\alpha \in R$ satisfying \eqref{eq:givenInfiniteInequality} and $(y_1, y_2, y_3) \in R^3$ satisfying
\begin{align*}
 ay_1^2 + by_2^2 - aby_3^2 &= \alpha^2(\kappa \beta)^2 + 2\alpha \\
\Leftrightarrow (\alpha (\kappa \beta)^2 + 1)^2 - a(\kappa \beta y_1)^2 - b(\kappa \beta y_2)^2 + ab(\kappa \beta y_3)^2 &= 1.
\end{align*}
Let
\[
 x = (\alpha (\kappa \beta)^2 + 1) + (\kappa \beta y_1)i + (\kappa \beta y_2)j + (\kappa \beta y_3)k,
\]
so $\Nm_A(x) = 1$. Also, $x' := \alpha\kappa^2\beta + (\kappa y_1)i + (\kappa y_2)j + (\kappa y_3)k \in R + \kappa Ri + \kappa Rj + \kappa Rk \subset \cO$, so $x = \beta x' + 1 \equiv 1 \bmod I\cO$. Hence $x \in \Gamma(I)$ and $x \neq \pm 1$ with
\[
 2|\alpha (\kappa \beta)^2 + 1| < 4|\alpha||\kappa\beta|^2 \le 4C_{\sigma_1}\Nm(\kappa)^2\Nm(\beta)^2 \le 4C_{\sigma_1}\Nm(\kappa)^2B_K^2\Nm(I)^2. \qedhere
\]
\end{proof}

\section{Graphs of Large Girth} \label{sec:Girth}
As we alluded to earlier, the large systole problem has an analogue for regular graphs. This final section briefly compares the problems in these two settings. For more details, see for example \cite{MR2532876} and references and therein.

The \emph{girth} of an undirected graph $X$, denoted $\girth(X)$, is the length of the shortest cycles (closed path without backtracking). This is the analogue of systole and noncontractible loops. For Riemann surfaces, recall that genus was area, up to a constant factor irrelevant for the systole problem. The analogous notion for a graph $X$ is simply the number of vertices, denoted $|X|$. Thus the following definitions are analogous to those for Riemann surfaces.
\begin{defn}
 A family of finite $k$-regular graphs $X_i$ with $|X_i| \rightarrow \infty$ is \emph{of large girth} if there exists constants $\gamma > 0$ and $c$, independent of $i$, such that
 \[
  \girth(X_i) \ge \gamma\log|X_i| - c.
 \]
\end{defn}
\begin{problem}[The large girth problem (for $k$-regular graphs)] 
Determine the supremum of $\gamma$ such that there exists a family of $k$-regular graphs of large girth with this $\gamma$. 
\end{problem}
As with systoles, it is easy to prove the upper bound $\gamma \le 2$. Moreover, the current best known constant of $\gamma = \frac{4}{3}$ was first achieved by an arithmetic construction, also using quaternions, by Lubotzky, Phillips, and Sarnak \cite{MR963118}. The truth of the girth problem is thus also known to lie within $\frac{4}{3} \le \gamma \le 2$.

The compact Riemann surfaces of the Buser-Sarnak construction were quotients of the upper half plane by principal congruence subgroups of a Fuchsian derived from a quaternion algebra. Similarly, the Ramanujan graphs of \cite{MR963118} are obtained by starting with an infinite tree associated to the integral Hamiltonian quaternion algebras, then quotienting at principal congruence subgroups to obtain finite graphs, identified with Cayley graph of $\PGL_2(\Z/q\Z)$. Generalizations of this construction have also achieved $\gamma = \frac{4}{3}$ for regular graphs of various degrees. A natural generalization using octonions was considered in \cite{DT}. However, due to an error found in the paper, the construction does not yield the claimed $\gamma > \frac{4}{3}$. Unlike in the case of Riemann surfaces, Gamburd et al. showed in \cite{MR2532876} that random Cayley graphs of certain families of groups are of large girth, achieving $\gamma = 1$ in the best case.

Biggs and Boshier in \cite{MR1064675} showed that $\gamma = \frac{4}{3}$ is the best possible value of $\gamma$ for the Ramanujan graphs of \cite{MR963118}. They obtain an upper bound on girth using, as we did, representation by a ternary quadratic form. The form obtained from the Hamiltonian quaternion algebra $(\frac{-1, -1}{\Q})$ is the sum of three squares. For positive definite forms, spinor genus theory no longer applies, so local-global theory is more difficult than for indefinite forms. However, this particular form is known to lie in a genus consisting of a single class, and Biggs and Boshier directly used the consequent Legendre's theorem on sums of three squares. For constructions based on more general quaternions, we must appeal to the recent results on the equidistribution of global solutions across genera for representation by positive definite ternary forms, first obtained by Cogdell, Piatetski-Shapiro, and Sarnak \cite{CPSS}. For the latest treatment with an improved constant, see Blomer and 
Harcos \cite{MR2647133}.

\bibliographystyle{alpha}
\bibliography{refs.bib}

\end{document}